\documentclass[12pt]{article}
\date{}

\usepackage[utf8]{inputenc}
\usepackage[english]{babel}
\usepackage{amssymb}
\usepackage{amsthm}
\usepackage{enumerate}
\usepackage{enumitem}
\usepackage{amsmath}
\usepackage{setspace}
\usepackage{indentfirst}
\usepackage{titlefoot}
\usepackage{authblk}

\newcommand{\R}{\mathbb{R}}
\newcommand{\Z}{\mathbb{Z}}
\newcommand{\F}{\mathbb{F}}
\newcommand{\Q}{\mathbb{Q}}
\newcommand{\C}{\mathbb{C}}
\newcommand{\PG}{\textnormal{PG}}
\newcommand{\Spec}{\textnormal{Spec}}
\newcommand{\Cay}{\textnormal{Cay}}
\newcommand{\Irr}{\textnormal{Irr}}

\newcommand{\Supp}{\textnormal{supp}}
\newcommand{\abs}[1]{|#1|}

\newtheorem{theorem}{Theorem}[section]
\newtheorem{corollary}[theorem]{Corollary}
\newtheorem{lemma}[theorem]{Lemma}

\theoremstyle{definition}

\newtheorem{remark}{Remark}

\title{Continuous-time Quantum Walks on Cayley Graphs of Extraspecial Groups}
\author{Peter Sin\footnote{Research partially supported by a grant from the Simons Foundation (\#633214 to Peter Sin).}, Julien Sorci}

\AtEndDocument{\bigskip{\footnotesize%

	\textsc{Peter Sin, Department of Mathematics, University of Florida, P.O. Box 118105, Gainesville, FL, 32611, USA} \par  
 	 \textit{E-mail address:} \texttt{sin@ufl.edu} \par
 	 \addvspace{\medskipamount}
 	 \textsc{Julien Sorci, Department of Mathematics, University of Florida, P.O. Box 118105, Gainesville, FL, 32611, USA} \par
  	 \textit{E-mail address:} \texttt{jsorci@ufl.edu} \par
}}

\begin{document}
\maketitle

\begin{abstract}
	We study continuous-time quantum walks on normal Cayley graphs of certain non-abelian groups called extraspecial groups. By applying general results for graphs in association schemes we determine the precise conditions for perfect state transfer and fractional revival. Using partial spreads, we construct Cayley graphs on extraspecial $2$-groups that admit these  phenomena. We also show that there is no normal Cayley graph of an extraspecial group that admits instantaneous uniform mixing.
\end{abstract}

\section{Introduction}
\unmarkedfntext{2010 Mathematics Subject Classification. Primary 05C25; Secondary 81Q35 \newline\hspace*{2em}\textit{Key words and phrases}: Cayley graph, perfect state transfer, fractional revival, uniform mixing, extraspecial groups}
Analogous to a classical walk on a graph, a \textit{quantum walk} on a graph models the evolution of quantum states from interacting qubits in a quantum computer. The walk can be defined by a unitary matrix acting on the complex vector space with basis labeled by the vertices of the graph. If $X$ is a simple, undirected, loopless graph on $n$ vertices with adjacency matrix $A$, then the operator 
\[ U_A(t) := \sum_{k=0}^\infty \frac{(it)^k}{k!} A^k \]
acting on $\C^n$ defines a continuous-time quantum walk on the graph $X$ and is called the \textit{transition matrix} of the walk. There are several state transfer phenomena of interest related to the walk. We say that the graph $X$ admits \textit{perfect state transfer} from a vertex $u$ to a vertex $v$ at time $\tau \in \R^+$ with phase $\lambda \in \C$ if $u$ and $v$ are distinct and $U_A(\tau) e_u =\lambda e_v$. We say that $X$ admits \textit{fractional revival} from $u$ to $v$ at time $\tau \in \R^+$ if $u$ and $v$ are distinct and
\[ U_A(\tau)e_u= \alpha e_u + \beta e_v \]
for some $\alpha, \beta \in \C$ with $|\alpha|^2 + |\beta|^2 =1$, $\beta \neq 0$. Note that perfect state transfer is a special case of fractional revival when $\alpha=0$. Lastly, we say that $X$ admits \textit{instantaneous uniform mixing} at time $\tau \in \R^+$ if
\[ |U_A(\tau)|_{u,v} = \frac{1}{\sqrt{n}} \]
for all vertices $u$ and $v$ of $X$. Finding families of graphs whose quantum walks have these various properties is still an active area of research, and the current results suggest that they rarely occur \cite{CJLSYZ,GMR}. For example, Godsil showed that for each positive integer $k$ there are only finitely many connected graphs with maximum degree at most $k$ which admit perfect state transfer \cite[Corollary 10.2]{G}.
	
Owing in part to their well-understood spectra, one family of graphs which has received attention are Cayley graphs of finite groups. There are several papers on perfect state transfer in Cayley graphs of abelian groups. In \cite{B}, Ba\v{s}i\'{c} studied Cayley graphs of cyclic groups and gave necessary and sufficient conditions for perfect state transfer. Bernasconi et al. and Cheung et al. studied perfect state transfer in cubelike graphs \cite{BGS, CG}. Tan et al. generalized this work by studying Cayley graphs of arbitrary abelian groups in \cite{TFC}, and provided conditions for perfect state transfer in terms of the spectrum. 
	
While much work has been done for Cayley graphs of abelian groups, there are few results for non-abelian groups. Gerhardt and Watrous \cite{GW} considered Cayley graphs of symmetric groups. Cao and Feng considered normal Cayley graphs of dihedral groups and showed that such graphs must be integral and the order of the group must be even for perfect state transfer to occur \cite{CF2}. We will extend the theory of quantum walks on Cayley graphs by considering normal Cayley graphs of certain non-abelian groups, called extraspecial $2$-groups. We will primarily view a Cayley graph of a finite group as a graph in a particular association scheme, and use the existing literature to examine the quantum walks on these graphs. An important result of Chan et al. \cite{CCTVZ} showed that a necessary condition for perfect state transfer is for the group to contain a central involution. This motivates our choice of studying normal Cayley graphs of extraspecial $2$-groups, since these groups contain a unique central involution.
	
In section 2 we will present some preliminary results concerning perfect state transfer on graphs in association schemes, and we will determine when vertices in a normal Cayley graph are strongly cospectral. In section 3 we will provide the necessary background on extraspecial $p$-groups and their characters, and show that their normal Cayley graphs are integral. In section 4 we give a complete characterization of perfect state transfer in these graphs and use partial spreads to construct explicit examples. In section 5 we give precise conditions for fractional revival. In section 6 we prove that there is no normal Cayley graph of an extraspecial $p$-group which admits instantaneous uniform mixing. Lastly, in section 7 we give some concluding remarks and open questions.

\section{Preliminaries}

A commonly used method for examining the transition matrix $U_A(t)$ is to employ the \textit{spectral decomposition} of $A$, which we will now describe. If $\theta_0,...,\theta_r$ are the distinct eigenvalues of $A$ then the spectral decomposition of $A$ is $A=\sum_{i=0}^r \theta_i \tilde{E}_i$ where $\tilde{E}_i$ is the projection onto the eigenspace $\mathcal E(\theta_i)$. The projections $\tilde{E}_i$ are orthogonal idempotents, meaning that they satisfy
\begin{enumerate}[label=(\roman*)]
	\item $\tilde{E}_j^2 = \tilde{E}_j$ for all $j=0,...,r$.
	\item $\tilde{E}_i\tilde{E}_j=0$ whenever $i \neq j$.
	\item $\sum_{j=0}^r\tilde{E}_j=I$.
\end{enumerate}
Using these properties the transition matrix $U_A(t)$ can be expressed as
\[ U_A(t) = \sum_{j=0}^r e^{it\theta_j}\tilde{E}_j. \] 
There are some sets related to the spectral decomposition of $A$ which will be important for understanding perfect state transfer. The set $\Phi_u:=\{\theta_i : \tilde{E}_i e_u \neq 0 \}$ is called the \textit{eigenvalue support} of the vertex $u$. We say that vertices $u$ and $v$ are \textit{strongly cospectral} if 
\[ \tilde{E}_i e_u = \pm \tilde{E}_i e_v, \] 
for all $i=0,...,r$, and when $u$ and $v$ are strongly cospectral we define the sets
\[ \Phi_{uv}^+ := \{\theta_i : \tilde{E}_i e_u = \tilde{E}_i e_v \neq 0\}, \]
\[ \Phi_{uv}^- := \{ \theta_i : \tilde{E}_i e_u = - \tilde{E}_i e_v \neq 0 \}. \]
The sets $\Phi_{uv}^+$ and $\Phi_{uv}^-$ are disjoint by definition. The adjacency matrix of a normal Cayley graph enjoys certain algebraic properties, which we now describe. Let $G$ be a finite group with conjugacy classes $K_0 =\{1\}, K_1,...,K_d$, and let $R_i \subseteq G \times G$ be the relation $R_i = \{(g,h) \in G \times G : gh^{-1} \in K_i \}$. If $A_i$ denotes the adjacency matrix of the relation $R_i$, then the set of matrices $\mathcal A=\{A_0,A_1,...,A_d\}$ forms an \textit{association scheme}. In general this association scheme is symmetric if and only if the conjugacy classes of $G$ are closed under taking inverses. For an introduction to association schemes, see \cite{D}. Throughout the paper we will only consider \textit{normal} Cayley graphs, meaning that $S \subseteq G \setminus \{1\}$ will be a union of conjugacy classes, say $S=\cup_{i \in I}K_i$, where $I \subseteq \{1,2,...,d\}$. We will also assume that $S$ is closed under taking inverses, $S$ generates the group $G$, and $S \neq G \setminus \{1\}$. Then the normal Cayley graph $\Cay(G,S)$ is a connected, undirected graph in the association scheme $\mathcal A$ which is not the complete graph, and with adjacency matrix $A= \sum_{i \in I}A_i$. 
	
Coutinho et al. gave the following result for graphs in association schemes which admit perfect state transfer \cite[Theorem 4.1]{CGGV}. It first appeared in \cite[Theorem 4.1]{G}.
	 	 
\begin{theorem}\label{GCGV}
	Let $X$ be a graph belonging to an association scheme $\mathcal A= \{A_0, A_1,...,A_d\}$ and with adjacency matrix $A = \sum_{i \in I}A_i$ for some $I \subseteq \{0,1,...,d\}$. If $X$ admits perfect state transfer at time $\tau$, then there is a permutation matrix $T$ with no fixed points and of order two such that $U_A(\tau)=\lambda T$ for some $\lambda \in \C$. Moreover, $T$ is a class in the scheme. 
\end{theorem}

Theorem~\ref{GCGV} tells us that if the normal Cayley graph $\Cay(G,S)$ admits perfect state transfer, then there is some $i \in I$ for which $A_i$ is a permutation matrix of order $2$ with no fixed points. Therefore a necessary condition for perfect state transfer is that $G$ contains a central involution $z$. The matrix $T=A_i$ is the regular representation of the element $z$, and when perfect state transfer does occur it is from a vertex $g$ to $gz$, where $g \in G$.
	
The spectral decomposition of the adjacency matrix $A$
of $\Cay(G,S)$ can be described using the complex characters of $G$. We identify
the underlying vector space of the group algebra $\C G$ with the space of column vectors
on which $A$ acts, so that the characteristic vectors $e_x$, for $x\in G$, form the standard basis. The group $G$ also acts on $\C G$ through its left regular representation $L$.  For $g\in G$ the matrix $L(g)$ maps a basis vector $e_x$ to $e_{gx}$. 

Let $\Irr(G)$ denote the set of irreducible complex characters of $G$.
For $\chi\in\Irr(G)$ we let $\chi(S)=\sum_{s\in S}\chi(s)$. Then each  eigenvalue of $A$ is of the form  $\theta=\chi(S)/\chi(1)$
for some $\chi \in \Irr(G)$ and, conversely, each $\chi$ gives an eigenvalue.
Let $X_\theta= \{ \chi \in \Irr(G) : \chi(S)/\chi(1)=\theta \}$. Then each
$\chi\in X_\theta$ contributes $\chi(1)^2$ to the multiplicity of $\theta$.
For each eigenvalue $\theta$, let $\mathcal E(\theta)$ denote its eigenspace
and let $\tilde{E}_\theta$ denote the projection onto $\mathcal E(\theta)$.
Then we have the spectral decomposition $A= \sum_{\theta \in \Spec(A)} \theta \tilde{E}_\theta$.  The association scheme $\mathcal A$ of $G$ generates a semisimple algebra,
isomorphic to the center of the group algebra, which has an  basis of orthogonal idempotents
\begin{equation}\label{charidem}
	 E_\chi := \frac{\chi(1)}{|G|} \sum_{g \in G}\chi(g^{-1})L(g), \ \ \chi \in \Irr(G).
\end{equation}	
%ref1(2)
The columns of the idempotents $E_\chi$ are eigenvectors of $A$, since $AE_\chi =\frac{\chi(S)}{\chi(1)}E_\chi$.  Then the projection $\tilde{E}_\theta$ can be written as
\begin{equation}\label{projidem}
	\tilde{E}_\theta = \sum_{\chi \in X_\theta}E_\chi,
\end{equation}
	so that the spectral decomposition of $A$ is 
\begin{equation}\label{specdecomp}
	A = \sum_{\theta \in \Spec(A)} \theta \sum_{\chi \in X_\theta}E_\chi .
\end{equation}	
From (\ref{specdecomp}) we can show that the eigenvalue support of a vertex in $\Cay(G,S)$ contains each eigenvalue of the graph.

\begin{theorem}\label{eigensupport}
The eigenvalue support $\Phi_g$ of a vertex $g \in G$ in the graph $X=\Cay(G,S)$ contains each eigenvalue of $X$.
\end{theorem} 

\begin{proof}
Let $\theta \in \Spec(X)$. We will show that the column of $\tilde{E}_\theta$ labeled by $g \in G$ is nonzero. Using (\ref{charidem}) and (\ref{projidem}),  the $g^{th}$ column of $\sum_{\chi \in X_\theta}E_\chi$ is:
\begin{equation*}
\begin{split}
	\sum_{\chi \in X_\theta}E_\chi e_g &= \sum_{\chi \in X_\theta}\frac{\chi(1)}{|G|} \sum_{h \in G} \chi(h^{-1})e_{hg} \\
	&= \sum_{\chi \in X_\theta}\frac{\chi(1)}{|G|} \sum_{x \in G} \chi(gx^{-1})e_x. \\
\end{split}
\end{equation*}
The $g^{th}$ entry in this column is then $\sum_{\chi \in X_\theta}\frac{\chi(1)^2}{|G|}$, which is nonzero.
\end{proof}

While Theorem~\ref{GCGV} showed that if a normal Cayley graph admits perfect state transfer then transfer occurs between group elements of the form $g,gz$, for $g \in G$ and $z$ a central involution of $G$, we can alternatively use (\ref{specdecomp}) to show that if distinct vertices $g$ and $h$ in a normal Cayley graph $\Cay(G,S)$ are strongly cospectral, then $gh^{-1}$ must be a central involution.

\begin{theorem}\label{stronglycospectralcentralinvolution}
If distinct vertices $g,h \in G$ of the normal Cayley graph $X=\Cay(G,S)$ are strongly cospectral, then $gh^{-1}$ is a central involution of $G$.
\end{theorem}

\begin{proof}
Since $g$ and $h$ are strongly cospectral, then for $\theta \in \Phi_{gh}^+$ we have $\sum_{\chi \in X_\theta}E_\chi e_g = \sum_{\chi \in X_\theta}E_\chi e_h$. This implies
\begin{equation*}
	\begin{split}
		\sum_{\chi \in X_\theta}\frac{\chi(1)}{|G|} \sum_{x \in G} \chi(gx^{-1}) e_x &= \sum_{\chi \in X_\theta}\frac{\chi(1)}{|G|} \sum_{x \in G} \chi(hx^{-1})e_x. \\
	\end{split}
\end{equation*}
Then for each $x \in G$ we get
\begin{equation*}
	\begin{split}
		\sum_{\chi \in X_\theta}\frac{\chi(1)}{|G|} \chi(gx^{-1}) &= \sum_{\chi \in X_\theta}\frac{\chi(1)}{|G|} \chi(hx^{-1}). \\
	\end{split}
\end{equation*}
Setting $x=h$ we obtain
\begin{equation*}
	\begin{split}
		\sum_{\chi \in X_\theta}\frac{\chi(1)}{|G|} \chi(gh^{-1}) &= \sum_{\chi \in X_\theta}\frac{\chi(1)}{|G|} \chi(1). \\
	\end{split}
\end{equation*}
As $\abs{\chi(gh^{-1})}\leq \chi(1)$, this equation forces $\chi(gh^{-1})=\chi(1)$ for all $\chi\in X_\theta$. If $\mu \in \Phi_{gh}^-$ then a similar argument shows that $\psi(gh^{-1})=-\psi(1)$ for all $\psi \in X_\mu$.
	
Consequently, we have shown that $\chi(gh^{-1})=\pm\chi(1)$ for all $\chi \in \Irr(G)$. This means that $z:=gh^{-1}$ must represented by either the identity matrix or its negative
in every irreducible complex representation of $G$.  It follows that $z$ is a central
element with $z^2=1$.
\end{proof}

Following Theorem~\ref{stronglycospectralcentralinvolution} we can now determine when two vertices are strongly cospectral in $\Cay(G,S)$.
	
\begin{theorem}\label{stronglycospectral}
Let $G$ be a finite group with central involution $z \in G$. Vertices $g$ and $gz$ are strongly cospectral in $X=\Cay(G,S)$ if and only if for each $\chi, \psi \in \Irr(G)$, whenever $\chi(S)/\chi(1) = \psi(S)/\psi(1)$ it implies $\chi(z)/\chi(1)=\psi(z)/\psi(1)$.
\end{theorem}

\begin{proof}
First suppose that $g$ and $gz$ are strongly cospectral. Then for $\theta \in \Phi_{g,gz}^+$, we have $\sum_{\chi \in X_\theta}E_\chi e_g = \sum_{\chi \in X_\theta}E_\chi e_{gz}$ which implies
\begin{equation*}
	\begin{split}
		\sum_{\chi \in X_\theta}\frac{\chi(1)}{|G|} \sum_{x \in G} \chi(gx^{-1})e_x &= \sum_{\chi \in X_\theta}\frac{\chi(1)}{|G|} \sum_{x \in G} \chi(gzx^{-1})e_x. \\
	\end{split}
\end{equation*}
Then for each $x \in G$ we get
\begin{equation*}
	\begin{split}
		\sum_{\chi \in X_\theta}\frac{\chi(1)}{|G|} \chi(gx^{-1}) &= \sum_{\chi \in X_\theta}\frac{\chi(1)}{|G|} \chi(gzx^{-1}). \\
	\end{split}
\end{equation*}
If $\varphi$ is the representation affording $\chi$ then for all $h\in G$, we have
$\varphi(hz)=\varphi(h)\varphi(z)$. Here $\varphi(z)$ must be a scalar matrix
as $\varphi$ is irreducible and $z$ is central, and the scalar matrix must be $\pm I$ as $z$ is an involution. Taking traces, we see that the sign is $\chi(z)/\chi(1)$.
Setting $x=g$ in the above equation yields
\begin{equation*}
\begin{split}
	\sum_{\chi \in X_\theta}\frac{\chi(1)}{|G|} \chi(1) &= \sum_{\chi \in X_\theta}\frac{\chi(1)}{|G|} \frac{\chi(z)}{\chi(1)} \chi(1). \\
\end{split}
\end{equation*}
From this we conclude that $\chi(z)/\chi(1)=1$ for all $\chi \in X_\theta$. A similar argument shows that if $\mu \in \Phi_{g,gz}^-$ then for all $\psi \in X_\mu$ we have $\psi(z)/\psi(1)=-1$.
	
Conversely, suppose that for each $\chi, \psi \in \Irr(G)$, whenever $\chi(S)/\chi(1) = \psi(S)/\psi(1)$ it implies $\chi(z)/\chi(1)=\psi(z)/\psi(1)$. Then for $\theta \in \Spec(X)$ with $\chi(z)/\chi(1)=1$ for all $\chi \in X_\theta$ we have
\begin{equation*}
\begin{split}
	\sum_{\chi \in X_\theta}E_\chi e_{gz} &=\sum_{\chi \in X_\theta}\frac{\chi(1)}{|G|} \sum_{x \in G} \chi(gzx^{-1})e_x \\
	&=\sum_{\chi \in X_\theta}\frac{\chi(1)}{|G|} \sum_{x \in G} \frac{\chi(z)}{\chi(1)} \chi(gx^{-1})e_x  \\
	&=\sum_{\chi \in X_\theta}\frac{\chi(1)}{|G|} \sum_{x \in G} \chi(gx^{-1})e_x  \\
	&= \sum_{\chi \in X_\theta}E_\chi e_g, \\
\end{split}
\end{equation*}
with a similar computation if $\chi(z)/\chi(1)=-1$ for all $\chi \in X_\theta$.
\end{proof}

\begin{remark}\label{eigpartitions}
	From Theorem~\ref{stronglycospectral} and its proof it is clear that if the vertices $g$ and $gz$ are strongly cospectral in $\Cay(G,S)$ then the sets $\Phi_{g,gz}^+$ and $\Phi_{g,gz}^-$ are
\[ \Phi_{g,gz}^+ = \{ \chi(S)/\chi(1) : \chi \in \Irr(G) \ \textnormal{and} \ \chi(z)/\chi(1)=1 \} \]
\[ \Phi_{g,gz}^- = \{ \psi(S)/\psi(1) : \psi \in \Irr(G) \ \textnormal{and} \ \psi(z)/\psi(1)=-1 \} \]
Note that the sets $\Phi_{g,gz}^+$ and $\Phi_{g,gz}^-$ do not depend on the choice of group element $g$. 
\end{remark}

\section{Extraspecial Groups}

A finite group $G$ is said to be an \textit{extraspecial p-group} if $G$ is a $p$-group whose center $Z$ has order $p$, and $G/Z$ is isomorphic to $(\Z/p\Z)^m$ for some $m \geq 1$. It is well known that an extraspecial $p$-group must have order $p^{2n+1}$ for some $n \geq 1$. The irreducible characters of extraspecial $p$-groups are well studied. If $G$ is an extraspecial $p$-group of order $p^{2n+1}$ then the irreducible linear characters of $G$ are the linear characters of $(\Z/p\Z)^{2n}$. The irreducible non-linear characters are
\[ \psi_\lambda(g)=\begin{cases}
		0, & \text{if $g \notin Z$} \\
		p^n \lambda(g), &\text{if $g \in Z$.} \\
	\end{cases} \] 
where $\lambda$ is a non-principal character of $Z$. In particular, when $G$ is an extraspecial $2$-group of order $2^{2n+1}$ the irreducible linear characters are $\chi_y(g)=(-1)^{y \cdot \bar{g}}$, where $y \in (\Z/2\Z)^{2n}$ and $\bar{g}$ is the image of $g \in G$ in $(\Z/2\Z)^{2n}$. Letting $z$ denote the unique central involution of $G$, then the unique irreducible non-linear character is
\begin{equation}\label{nonlinearchar}
	\psi(g) = \begin{cases}
	0, & \text{if $g \in G \setminus Z$} \\
	-2^n, & \text{if $g = z$} \\
	2^n, & \text{if $g = 1$.} \\
        \end{cases}
\end{equation}
For more background on extraspecial groups, see \cite[Chapter 5.5]{G3}. 

From now until the end of section \ref{fracrev} we will let $G$ denote an extraspecial $2$-group of order $2^{2n+1}$ with unique central involution $z$. We note that for each positive integer $n$ there are two isomorphism types of extraspecial $2$-groups of order $2^{2n+1}$; our results do not depend on the isomorphism type. We will let $S \subseteq G \setminus \{1\}$ be a union of conjugacy classes that generates the group $G$. Note that each conjugacy class in $G$ is closed under inversion, and that $S$ generates $G$ if and only if the image of $S$ in $(\Z/2\Z)^{2n}$ generates $(\Z/2\Z)^{2n}$. It is easy to see that the conjugacy classes of $G$ with size greater than $1$ are the non-identity cosets in the quotient $G/Z$ and thus have size $2$. Explicitly, let us write $S \setminus \{z\} = K(x_1)\cup K(x_2) \cup\cdots\cup K(x_\ell)$,  where $K(x_i)$ is a conjugacy class of size $2$ of the element $x_i \in G$, and $z$ is possibly not in $S$. As $G/Z$ is isomorphic to $(\Z/2\Z)^{2n}$, we will let $\bar{g}$ denote the image of a group element $g \in G$ in $(\Z/2\Z)^{2n}$. Additionally, for $y \in (\Z/2\Z)^{2n}$ we will let $e_y$ denote the number of elements $x_i$ with $y \cdot \bar{x}_i = 0$ modulo $2$.

\begin{theorem}\label{spec} In the notation above, let $X=\Cay(G,S)$ be the normal
  Cayley graph with connecting set $S$. Then $X$ is integral, and its spectrum is (the multiset):
 \[ \Spec(X) = \{ 4e_y -2\ell+1: y \in (\Z/2\Z)^{2n}\} \cup \{ -1\}, \]
	when $z \in S$, and 
 \[ \Spec(X) = \{ 4e_y -2\ell: y \in (\Z/2\Z)^{2n}\} \cup \{ 0\}, \]
	when $z \notin S$.
\end{theorem}

\begin{proof}
We will prove the case when $z \in S$ as the case when $z \notin S$ follows similarly. The eigenvalues of $X$ are $\chi(S)/\chi(1)$ where $\chi$ is an irreducible character of $G$. First suppose that $\chi$ is an irreducible linear character. Then $\chi(x)=(-1)^{y \cdot \bar{x}}$ for some $y \in (\Z/2\Z)^{2n}$, and the eigenvalue afforded by $\chi$ is computed by
\begin{equation*}
	\begin{split}
		\chi(S) &= 1+ 2 \sum_{i=1}^\ell (-1)^{y \cdot \bar{x}_i} \\
		&= 1+2(e_y-\ell+e_y) \\ 
		&= 4e_y-2\ell+1,\\
	\end{split}
\end{equation*}
giving the eigenvalues in the first set above. Next suppose that $\psi$ is an irreducible nonlinear character.%ref1(3)
Then from equation~(\ref{nonlinearchar}), $\psi$ has degree $2^n$ and vanishes on the conjugacy classes of size two. Thus the eigenvalue afforded by $\psi$ is 
\[ \psi(S)/\psi(1)=\psi(z)/\psi(1)=-1 \]
as claimed. 
\end{proof}

\begin{theorem}\label{stronglycospectralG}
Let $X=\Cay(G,S)$ be a normal Cayley graph where $G$ is an extraspecial $2$-group. Vertices $g$ and $gz$ of $X$ are strongly cospectral if and only if the sets 
\[ \{ \chi(S)/\chi(1): \chi \in \Irr(G) \ \textnormal{linear}\}, \ \ \{ \psi(S)/\psi(1) : \psi \in \Irr(G) \ \textnormal{nonlinear}\} \]
are disjoint. 
\end{theorem}

\begin{proof}%ref1(4)
If $g$ and $gz$ are strongly cospectral then from Remark~\ref{eigpartitions} we have
\[ \Phi_{g,gz}^+= \{  \chi(S)/\chi(1): \chi \in \Irr(G) \ \textnormal{linear}\} \]
and
\[ \Phi_{g,gz}^- = \{\psi(S)/\psi(1) : \psi \in \Irr(G) \ \textnormal{nonlinear}\}, \]
which are disjoint sets by definition. Conversely, if $\{ \chi(S)/\chi(1): \chi \in \Irr(G) \ \textnormal{linear}\}$ and $\{\psi(S)/\psi(1) : \psi \in \Irr(G) \ \textnormal{nonlinear}\}$ are disjoint, then it is never the case that $\chi(S)/\chi(1) = \psi(S)/\psi(1)$ for some linear character $\chi$ of $G$ and some nonlinear character $\psi$ of $G$. Thus Theorem~\ref{stronglycospectral} implies that $g$ and $gz$ are strongly cospectral.
\end{proof}

\section{Perfect State Transfer}
We are now ready to give a preliminary result on perfect state transfer in normal Cayley graphs of extraspecial $2$-groups. Since the spectrum of $\Cay(G,S)$ will depend on whether or not $z$ is in $S$, we will consider each case separately. For a rational number $x \in \Q$ we will let $\nu_2(x)$ denote its $2$-valuation. We will use the following characterization of perfect state transfer in regular graphs, which is a reformulation of a theorem given by Coutinho \cite[Theorem 2.4.4]{Co}. 

\begin{theorem}\label{thmcoutinhopst}
	A $k$-regular graph $X$ admits perfect state transfer between vertices $u$ and $v$ at time $\tau$ with phase $\lambda$ if and only if the following hold:
\begin{enumerate}[label=(\roman*)]
	\item $u$ and $v$ are strongly cospectral.
	\item The nonzero elements of $\Phi_u$ are all integers.
	\item There is an integer $N \geq 0$ such that $N = \nu_2(k - \mu)$ for all $\mu \in \Phi_{uv}^-$.
	\item With $N$ given in (iii), each $\theta \in \Phi_{uv}^+$ has $\nu_2(k - \theta) > N$.
\end{enumerate}
	Moreover, if these conditions hold and perfect state transfer occurs between $u$ and $v$ at time $\tau$ with phase $\lambda$, then:
\begin{enumerate}[label=(\alph*)]
	\item  there is a minimum time $\tau_0$ where $X$ admits perfect transfer, which is
\[ \tau_0 = \frac{\pi}{f}, \ \ f = \gcd\{\theta_0 - \theta:\theta \in \Phi_u \} \]
	\item $\tau$ is an odd multiple of $\tau_0$.
	\item $\nu_2(f)=N$.
\end{enumerate}	
\end{theorem}

	We have the following criteria for perfect state transfer when $z$ is in $S$.

\begin{lemma}\label{pstz} In the notation of Theorem~\ref{spec},
 the graph $X=\Cay(G,S)$ with $z \in S$ admits perfect state transfer if and only if:
\begin{enumerate}[label=(\roman*)]
	\item $e_y \neq \frac{\ell-1}{2}$ for each $y \in (\Z/2\Z)^{2n}$;
	\item $\nu_2(\ell-e_y) \geq \nu_2(\ell+1)$ for each $y \in (\Z/2\Z)^{2n}$.
\end{enumerate}
\end{lemma}

\begin{proof}
	Suppose first that $X$ admits perfect state transfer from $g$ to $gz$. From Theorem~\ref{eigensupport} we know that the eigenvalue supports of each vertex contain every eigenvalue of $X$. Theorem~\ref{thmcoutinhopst} implies that the vertices $g$ and $gz$ are strongly cospectral, and by Remark~\ref{eigpartitions} the sets $\Phi_{g,gz}^+$ and $\Phi_{g,gz}^-$ are
	\[ \Phi_{g,gz}^+= \{ 4e_y-2\ell+1 : y \in (\Z/2\Z)^{2n}  \}, \]
	\[ \Phi_{g,gz}^-= \{-1\}. \] 
	As the sets $\Phi_{g,gz}^+$ and $\Phi_{g,gz}^-$ are disjoint, then $4e_y-2\ell+1 \neq -1$, which implies (i). Moreover, the valuation criteria of Theorem~\ref{thmcoutinhopst} (iii)-(iv) imply 
	\[ \nu_2(|S|-4e_y +2\ell-1) > \nu_2(|S|+1) \]
	for all $y \in (\Z /2\Z)^{2n}$. Simplifying, we get
	\[\nu_2(4\ell - 4e_y) > \nu_2(2\ell+2)\]
	which is equivalent to
	\[ \nu_2(\ell-e_y) \geq \nu_2(\ell+1) \]
	so that (ii) is necessary. Conversely, if (i)-(ii) hold then (i) implies that the sets 
	\[ \{ 4e_y-2\ell+1 : y \in (\Z/2\Z)^{2n}\}, \ \ \{-1\} \] 
	are disjoint, so by Theorem~\ref{stronglycospectralG} the vertices $g$ and $gz$ are strongly cospectral. From Theorem~\ref{eigensupport}, the eigenvalue support of $g$ contains each eigenvalue of $X$, and each such eigenvalue is an integer by Theorem~\ref{spec} so that Theorem~\ref{thmcoutinhopst} (ii) is satisfied. Lastly, (ii) implies that 
	\[\nu_2(|S|-4e_y+2\ell-1) > \nu_2(|S|+1)\]
	so that Theorem~\ref{thmcoutinhopst} (iii)-(iv) hold. Therefore $X$ admits perfect state transfer from $g$ to $gz$.  
\end{proof}

	Note that in Lemma~\ref{pstz}, (ii) implies (i). Next we will consider the case when $z \notin S$ and prove a preliminary result analogous to Lemma~\ref{pstz}.
	
\begin{lemma}\label{pstnoz}
	In the notation of Theorem~\ref{spec},
 the graph $X=\Cay(G,S)$ with $z \notin S$ admits perfect state transfer if and only if:
\begin{enumerate}[label=(\roman*)]
	\item $e_y \neq \frac{\ell}{2}$ for each $y \in (\Z/2\Z)^{2n}$;
	\item $\nu_2(\ell-e_y) \geq \nu_2(\ell)$ for each $y \in (\Z/2\Z)^{2n}$.
\end{enumerate}
\end{lemma}

\begin{proof}
  The proof is analogous to the proof of Lemma~\ref{pstz}. First suppose that $X$ admits perfect state transfer between vertices $g$ and $gz$. Then $g$ and $gz$ are strongly cospectral, and we must have $4e_y-2\ell \neq 0$ since  %ref1(5)
$\Phi_{g,gz}^+=\{4e_y-2\ell : y \in (\Z/2\Z)^{2n} \}$ and $\Phi_{g,gz}^- = \{0\}$ are disjoint. This gives (i). For (ii), note that the valuation criteria of Theorem~\ref{thmcoutinhopst} (iii)-(iv) tell us
\begin{equation*}
	\begin{split}
		\nu_2(2\ell- 4e_y+2\ell) & > \nu_2(2\ell) \\
		\nu_2(4\ell-4e_y) & > \nu_2(2\ell) \\
	\end{split}
\end{equation*}
	as claimed. 
	
	Conversely, if (i) and (ii) hold then (i) implies that the vertices $g$ and $gz$ are strongly cospectral and (ii) gives the required $2$-valuation criteria. 
\end{proof}

	As with Lemma~\ref{pstz}, we note that in Lemma~\ref{pstnoz}, (ii) implies (i). We summarize the statements of Lemma~\ref{pstz} and Lemma~\ref{pstnoz} into one theorem.

\begin{theorem}\label{pstmain}
	In the notation of Theorem~\ref{spec}, the graph $X=\Cay(G,S)$ admits perfect state transfer at time $\tau$ if and only if either of the following hold:
 \begin{enumerate}[label=(\roman*)]
 	\item $z \in S$ and $\nu_2(\ell - e_y) \geq \nu_2(\ell+1)$ for all $y \in (\Z/2\Z)^{2n}$; or
 	\item $z \notin S$ and  $\nu_2(\ell-e_y) \geq \nu_2(\ell)$ for all $y \in (\Z/2\Z)^{2n}$.
 \end{enumerate}
 If (i) holds then $\tau$ is an odd multiple of $\pi/d$, where \[d= \gcd \big( \{2\ell+2\} \cup \{ 4e_y -4\ell : y \in (\Z/2\Z)^{2n}\} \big)\] and if (ii) holds then $\tau$ is an odd multiple of $\pi/c$, where
\[c= \gcd \big( \{2\ell\} \cup \{ 4e_y -4\ell : y \in (\Z/2\Z)^{2n}\} \big). \] \qed
\end{theorem}

	Clearly the conditions of Theorem~\ref{pstmain} hold when $z \in S$ and $\ell$ is even, or when $z \notin S$ and $\ell$ is odd. In the case where $z \notin S$ and $\Cay(G,S)$ admits perfect state transfer, then the transfer occurs between nonadjacent vertices $g$ and $gz$, where $g \in G$. 
	
	In this case, we remark that the distance between $g$ and $gz$ is two. For if $x$ is a noncentral element in $S$ then $x$ has order two or four. Since $S$ is closed under conjugation then $xz$ is in $S$ as well. If $x$ has order two then there is a path $g, gx, gx(xz)=gz$ in $\Cay(G,S)$. If $x$ has order four then $x^2=z$ and we have a path $g,gx,gx^2=gz$. 

	We can say more about the minimum times $\pi/d$ and $\pi/c$ given in Theorem~\ref{pstmain}. Namely, we will show that both the integers $c$ and $d$ are powers of $2$. A similar result was given in \cite[Lemma 4.2]{TFC}, although we we will prove a more general result here. The proof makes use of some results from the representation theory of general linear groups, which we will review now. We refer the reader to \cite[Section 3(C)]{M} for more details. Let $P$ denote the points of  $\PG(m-1,q)$ for some prime power $q$ of a prime $p$,
%ref1(6)
and let $F=\F_r$ for a prime $r\neq p$. The $F$-vector space $F^P$ of $F$-valued functions of $P$ has a basis of characteristic functions of points, where for any subset $A$ of $P$ the characteristic function of $A$ is the function $\delta_A : P \rightarrow F$ taking the value $1$ on elements of $A$ and $0$ elsewhere.
	
The general linear group $\textnormal{GL}(m,q)$ acts on $F^P$ by the action defined on the characteristic functions of points
%corrected 7/19/22 PS
	\[ M \cdot \delta_{\{a\}}(x)=\delta_{\{a\}}(xM), \ M \in \textnormal{GL}(m,q)   \]
	and extending linearly, making $F^P$ a module for the group algebra $F \textnormal{GL}(m,q)$. The structure of this module is well-known and depends on the choice of prime $r$. When $r$ does not divide $\abs{P}$, then $F^P$ decomposes as the direct sum 
\begin{equation}\label{pnotdivideorder}
	F \mathbf{1} \oplus W,
\end{equation}
	where $\mathbf{1}$ is the characteristic function of $P$, and $W$ is the irreducible submodule spanned by the functions $f \in F^P$ satisfying $\sum_{x \in P} f(x) = 0$, referred to as the \textit{augmentation submodule}. When $r$ divides $\abs{P}$, then $F^P$ has a composition series 
\begin{equation}\label{pdivideorder}
	0 \subseteq F \mathbf{1} \subseteq W \subseteq F^P,
\end{equation}
	where $W$ and  $F \mathbf{1}$ are defined in the same way as in the previous case. 
	
	With this background in mind we can now turn to a preliminary lemma and its proof.	

\begin{lemma}\label{power2lem}
	Let $Y$ be a proper, nonempty subset of $P$, and suppose that $C$ is a $\textnormal{GL}(m,q)$-invariant class of subsets of $P$ such that $C$ contains at least one nonempty proper subset. Then the integer 
	\[ \gcd \big( |Y \cap H| : H \in C  \big) \]
	is a power of $p$. 
\end{lemma}

\begin{proof}
	Suppose that $r$ is some prime not equal to $p$ which divides $|Y \cap H|$ for all $H \in C$. Let $F=\F_r$, and consider the $F$-subspace $U = \langle \delta_H : H \in C \rangle$ of $F^P$. Note that $U \neq F \mathbf{1}$. For $H \in C$, the dot product of $\delta_H$ with $\delta_Y$ is $|Y \cap H|$. Since $r$ divides $|Y \cap H|$ for each $H \in C$ then this implies that $\delta_Y$ is in $U^\perp$.
	
	If $r$ does not divide $\abs{P}$, then following the decomposition of
        %ref1(3)
equation~(\ref{pnotdivideorder}) $U$ must contain the submodule $W$ since $U \neq F \mathbf{1}$.
But this implies that $U^\perp \subseteq F \mathbf{1}$, so that $\delta_Y$ is constant. This contradicts the hypothesis that $Y$ is a proper subset of $P$, and hence $r$ cannot be distinct from $p$.
If $r$ divides $\abs{P}$, then since $U \neq F\mathbf{1}$ the composition series of
%ref1(3)
equation~(\ref{pdivideorder}) implies that $W \subseteq U$. But then we get that $U^\perp \subseteq W^\perp =F \mathbf{1}$, hence $\delta_Y$ is in $F \mathbf{1}$. This contradicts the assumption that $Y$ is a proper subset of $P$, and again $r$ cannot be distinct from $p$.
	 From both of these cases we conclude that $\gcd(|Y \cap H| : H \in C)$ is a power of $p$, as claimed.  
\end{proof}

	From Lemma~\ref{power2lem} it now easily follows that the integers $c$ and $d$ from Theorem~\ref{pstmain} are powers of $2$.

\begin{theorem}\label{power2}
	In the notation of Theorem~\ref{pstmain}, the integers
	\[ d=\gcd \big( \{2\ell+2\} \cup \{ 4e_y -4\ell : y \in (\Z/2\Z)^{2n}\}\big) \]
	and
	\[ c=\gcd\big( \{2\ell\} \cup \{ 4e_y -4\ell : y \in (\Z/2\Z)^{2n}\}\big) \]
	are both powers of $2$.
\end{theorem}

\begin{proof}
	We show that the integer 
	\[ M = \gcd \big( \ell - e_y : y \in (\Z/2\Z)^{2n} \big)\]
	is a power of $2$, from which it follows that $d$ and $c$ are both powers of $2$. Let $C$ be the class of subsets
	\[ C = \Big \{ \{ x : y \cdot x =1 \} : y \in \F_2^{2n} \setminus \{0\}   \Big \} \]
	of $\PG(2n-1,2)$. Then $C$ is $\textnormal{GL}(2n,2)$-invariant, and so the integer
	\[ N=\gcd( |\{\bar{x}_1,...,\bar{x}_\ell\} \cap H | : H \in C ) \] 
	is a power of $2$ by Lemma~\ref{power2lem}. If $H$ is the subset $\{ x : y \cdot x =1 \}$ then $|\{\bar{x}_1,...,\bar{x}_\ell\} \cap H | = \ell - e_y$ so that $N=M$ and therefore $M$ is a power of $2$. 
\end{proof}

	From Theorem~\ref{power2} we can simplify the minimum time for perfect state transfer given in Theorem~\ref{pstmain}. We state the simplified theorem for the convenience of the reader. 
	
\begin{theorem}\label{pstmain2}
	In the notation of Theorem~\ref{spec}, the graph $X=\Cay(G,S)$ admits perfect state transfer at time $\tau$ if and only if either of the following hold:
 \begin{enumerate}[label=(\roman*)]
 	\item $z \in S$ and $\nu_2(\ell - e_y) \geq \nu_2(\ell+1)$ for all $y \in (\Z/2\Z)^{2n}$; or
 	\item $z \notin S$ and  $\nu_2(\ell-e_y) \geq \nu_2(\ell)$ for all $y \in (\Z/2\Z)^{2n}$.
 \end{enumerate}
 	If (i) holds then $\tau$ is an odd multiple of $\pi/2^{\nu_2(2\ell +2)}$ and if (ii) holds then $\tau$ is an odd multiple of $\pi/2^{\nu_2(2\ell)}$. \qed
\end{theorem}
	
We can also show that when $\Cay(G,S)$ is not the complete graph, both $\Cay(G,S)$ and its complement admit perfect state transfer at the same time. For the following theorem
%ref1(7)
we let $S^c$ denote the complement of the subset $S$ in $G \setminus \{1\}$, and we will let $\bar{A}$ be the adjacency matrix of $\Cay(G,S^c)$. 

\begin{theorem}\label{comppst}
Let $G$ be any finite group. Suppose the normal Cayley graph $X=\Cay(G,S)$ admits perfect state transfer from $g$ to $gz$, for all $g \in G$, for some central involution $z \in G$, and at time $\tau$. Then $\bar{X} = \Cay(G,S^c)$ admits perfect state transfer from $g$ to $gz$ for all $g \in G$ and at time $\tau$ if and only if $|G|\tau$ is an integer multiple of $2 \pi$.
\end{theorem}

\begin{proof}
Since $X$ admits perfect state transfer, then from Theorem~\ref{thmcoutinhopst} we can write the transition matrix of $X$ at the time $\tau$ as
\[ e^{i\tau A} = \lambda T \]
for some $\lambda \in \C$, where $T$ is the regular representation of $z$. The adjacency matrix of the complementary graph $\bar{X}$ is related to $A$ by 
\[ \bar{A}=J-I-A. \]
Since $T$ has order two, then the transition matrix for $\bar{X}$ at time $\tau$ is
\[ e^{i \tau \bar{A}} = \lambda^{-1} e^{i \tau (J-I)}T.  \]
Therefore $\bar{X}$ admits perfect state transfer if and only if $e^{i\tau (J-I)}=\mu I$ for some $\mu \in \C$. The eigenvalues of $J-I$ are $|G|-1$ and $-1$, so if $J-I$ has spectral decomposition $(|G|-1)E_1-E_2$, then 
\[ e^{i\tau (J-I)}=e^{i\tau(|G|-1)}E_1 +e^{-i \tau}E_2. \]
It then follows that $e^{i\tau (J-I)}=\mu I$ if and only if $e^{i\tau(|G|-1)}=e^{-i \tau}$, if and only if $|G|\tau$ is an integer multiple of $2\pi$.
\end{proof}

\begin{corollary}
Let $G$ be an extraspecial $2$-group of order $2^{2n+1}$, and suppose that $S$ is a proper subset of $G \setminus \{1\}$. Then $X=\Cay(G,S)$ admits perfect state transfer at time $\tau$ if and only if the complementary graph $\bar{X}=\Cay(G,S^c)$ admits perfect state transfer at time $\tau$.
\end{corollary}

\begin{proof}
	Suppose that $X=\Cay(G,S)$ admits perfect state transfer at time $\tau$ and that $z \in S$. From Theorem~\ref{thmcoutinhopst}, it suffices to check that the theorem is true when $\tau$ equals the minimum time $\tau_0$, so we assume that $\tau = \tau_0$.  Following Theorem~\ref{comppst}, we show that $2^{2n+1}\tau$ is an integer multiple of $2\pi$. From Theorem~\ref{pstmain2} the minimum time is $\tau =\pi/2^{\nu_2(2\ell+2)}$. Since $S$ is a proper subset of $G \setminus \{1\}$, then $\ell$ is at most $2^{2n}-2$, and hence $\ell + 1 \leq 2^{2n}-1$. Thus we have
	\[ 2^{\nu_2(2\ell+2)} \leq 2\ell + 2 \leq 2^{2n+1}-2 \]
	which implies $\nu_2(2\ell+2) < 2n+1$. Therefore $2^{2n+1}\pi/2^{\nu_2(2\ell+2)}$ is indeed an integer multiple of $2\pi$, as claimed. 
	
	Suppose now that $\Cay(G,S)$ admits perfect state transfer at time $\tau$ and that $z \notin S$. Arguing similarly to the case where $z \in S$, we know that $\tau=\pi /2^{\nu_2(2\ell)}$.
	Again we get the inequalities
	\[ 2^{\nu_2(2\ell)} \leq 2 \ell \leq 2^{2n+1}-2 \]
	so that $\nu_2(2\ell) < 2n+1$. Therefore $2^{2n+1}\pi/2^{\nu_2(2\ell)}$ is an integer multiple of $2\pi$.
\end{proof}

	As we noted after Theorem~\ref{pstmain}, there are two particularly simple choices of connecting set $S$ where $\Cay(G,S)$ admits perfect state transfer. Namely, we can choose $S$ so that $z \in S$ and $\ell$ is even, or choose $S$ with $z \notin S$ and $\ell$ is odd. We have  less obvious examples where $S$ does not contain $z$ and $\ell$ is even. By taking complements, it also gives graphs admitting perfect state transfer whose connecting set contains $z$ and $\ell$ is odd.
	
	A \textit{partial k-spread} in $\F_q^n$ is a set of $k$-dimensional subspaces of $\F_q^n$ such that any two intersect in the zero subspace. Elementary counting shows that a partial $k$-spread has at most $(q^n-1)/(q^k-1)$ elements. 
	
\begin{theorem}\label{spread}
	Let $T$ be a partial $k$-spread in $\F_2^{2n}$ with $N$ elements, and assume that the points of $T$ span $\F_2^{2n}$. In the notation of Theorem~\ref{spec}, let $X=\Cay(G,S)$ with $S= K(x_1) \cup\cdots\cup K(x_\ell)$ such that $\bar{S}$ is the set of points incident to the elements of $T$. If $\nu_2(N) \leq k-1$, then $X$ is a connected graph which admits perfect state transfer with minimum time $\pi /2^{\nu_2(2N)}$.
	
	In particular, when $T$ is a partial $n$-spread for $n \geq 2$ and $N=2^{n-1}$ then $X$ is a connected graph which admits perfect state transfer with minimum time $\pi/2^n$. 
\end{theorem}

\begin{proof}
	We prove the statement first when $T$ is a partial $k$-spread with $N$ elements which span $\F_2^{2n}$. Since the points of $T$ are assumed to span $\F_2^{2n}$ then the graph $\Cay(G,S)$ is connected. To see that $\Cay(G,S)$ admits perfect state transfer, let $H$ be the hyperplane of $\F_2^{2n}$ defined by
	\[\{ x \in \F_2^{2n} : y \cdot x =0  \} \]
	for some $y \in \F_2^{2n} \setminus \{0\}$. Each element of $T$ is either contained in $H$ or intersects $H$ in a $(k-1)$-dimensional subspace. Let $M$ be the number of elements of $T$ contained in $H$. Counting the number of points of $\bar{S}$ that are contained in $H$ shows
\begin{equation*}
\begin{split}
	e_y &= M(2^k-1)+(N-M)(2^{k-1}-1)\\
		&= (M+N)2^{k-1} -N.\\
\end{split}
\end{equation*}
It then follows that $\ell - e_y = (N-M)2^{k-1}$, which implies that $\nu_2(\ell - e_y) \geq k-1$. Since $\ell = N(2^k -1)$, then $\nu_2(\ell)=\nu_2(N)\leq k-1$. From Theorem~\ref{pstmain2}, $X$ admits perfect state transfer with minimum time $\pi/2^{\nu_2(2N)}$.

	Lastly, when $T$ is a partial $n$-spread with $N=2^{n-1}$ then the points of $T$ span $\F_2^{2n}$ since $n \geq 2$. Therefore $\Cay(G,S)$ is a connected graph in this case, and the above computations prove that it admits perfect state transfer with minimum time $\pi/ 2^n$. 
\end{proof}
	
Note that the minimum time given in Theorem~\ref{spread} depends only on the number of elements in the spread. For example,  we can take $T$ to be a subset of the \textit{regular spread}, which is defined as follows. The vector spaces $\F_2^{2n}$ and $\F_{2^n}^2$ are isomorphic as $\F_2$-vector spaces. A $1$-dimensional subspace of $\F_{2^n}^2$ maps to an $n$-dimensional subspace of $\F_2^{2n}$, and the $2^{2n}+1$ images of the $1$-dimensional subspaces of $\F_{2^n}^2$ form a collection of  $n$-dimensional subspaces of $\F_2^{2n}$, with any two intersecting in the zero subspace, called the \textit{regular spread}.

\section{Fractional Revival}\label{fracrev}
A detailed description of fractional revival in graphs in association schemes was provided in Chan et al. \cite[Theorems 3.3, 3.5]{CCTVZ}. We can apply those results
to our setting of a normal Cayley graph $\Cay(G,S)$ of an extraspecial $2$-group. We use the notation from section 3.
Let $\alpha = \nu_2( \gcd( \ell - e_y : y \in (\Z/2\Z)^{2n}))$
and define integers $g$ and $h$ by:
	\[ g = \begin{cases}
		\min(2^{\nu_2(2\ell+2)}, 2^{\alpha+2}), & \textnormal{if} \ z \in S, \\
		\min(2^{\nu_2(2\ell)}, 2^{\alpha+2}), & \textnormal{if} \ z \notin S,\end{cases} \qquad  hg = 2^{\alpha+2}. \]
By Theorem~\ref{power2}  $g$ is equal to the integer in Theorem~\ref{pstmain} denoted by $d$ or $c$ according to whether or not $z\in S$.
In this notation, we can state the following special case of \cite[Theorem 3.5]{CCTVZ}. (The times for balanced fractional revival
are read off from \cite[Theorem 3.5]{CCTVZ}.)

\begin{theorem}\label{fracrevesp} Let $X=\Cay(G,S)$. 
 If $z \in S$ then:
\begin{enumerate}[label=(\alph*)]
	\item $\alpha +1 \leq \nu_2(\ell+1)$ if and only if $X$ admits neither perfect state transfer nor fractional revival.
	\item $\alpha \geq \nu_2(\ell+1)$ if and only if $X$ admits perfect state transfer. When $\alpha = \nu_2(\ell+1)$, $X$ admits perfect state transfer but no other form of fractional revival. 
	\item $\alpha \geq \nu_2(\ell+1)+1$ if and only if $X$ admits fractional revival that is different from perfect state transfer. In particular, $\frac{2\pi}{hg}$ is the minimum time when fractional revival occurs in $X$, and $X$ admits balanced fractional revival at time $\frac{\pi}{2g}$.
          \end{enumerate}
	If $z \notin S$ then:
\begin{enumerate}[label=(\roman*)]
	\item $\alpha +1 \leq \nu_2(\ell)$ if and only if $X$ admits neither perfect state transfer nor fractional revival.
	\item $\alpha \geq \nu_2(\ell)$ if and only if $X$ admits perfect state transfer. When $\alpha = \nu_2(\ell)$, $X$ admits perfect state transfer but no other form of fractional revival. 
	\item $\alpha \geq \nu_2(\ell)+1$ if and only if $X$ admits fractional revival that is different from perfect state transfer. In particular, $\frac{2\pi}{hg}$ is the minimum time when fractional revival occurs in $X$, and $X$ admits balanced fractional revival at time $\frac{\pi}{2g}$.
\end{enumerate}	\qed
\end{theorem}	
	
By slightly modifying Theorem~\ref{spread} and applying Theorem~\ref{fracrevesp}(iii) we have examples of sets $S \subseteq G \setminus \{1,z\}$ where $G$ is an extraspecial $2$-group and $\Cay(G,S)$ admits balanced fractional revival.
	
\begin{theorem}
	Let $T$ be a partial $k$-spread in $\F_2^{2n}$ with $N$ elements, and assume that the points of $T$ span $\F_2^{2n}$. Let $X=\Cay(G,S)$ with $S= K(x_1) \cup ... \cup K(x_\ell)$ such that $\bar{S}$ is the set of points incident to the elements of $T$. If $\nu_2(N) \leq k-2$, then $X$ is a connected graph which admits balanced fractional revival.
	
	In particular, when $T$ is a partial $n$-spread with $n \geq 3$ and $N=2^{n-2}$ then $X$ is a connected graph which admits balanced fractional revival.  \qed
\end{theorem}

\section{Uniform Mixing}

	We now turn to uniform mixing in normal Cayley graphs of extraspecial $p$-groups, where $p$ is any prime, not necessarily equal to $2$. For graphs in association schemes, Chan gave a useful characterization. Let $\mathcal A=\{A_0,A_1,...,A_d\}$ be an association scheme. The $\C$-algebra generated by $\mathcal A$ is called the \textit{Bose-Mesner algebra} of $\mathcal A$, and has an alternative basis of orthogonal idempotents, say $\{E_0,...,E_d\}$. Each $A_j$ can be written as a linear combination of the idempotents, say $A_j = \sum_{s=0}^d p_j(s)E_s$, and from the orthogonality of the idempotents it follows that the columns of the idempotents $E_s$ are eigenvectors of $A_j$ with eigenvalue $p_j(s)$.
	
	Recall that if $X$ is a graph on $n$ vertices in the association scheme $\mathcal A$ with adjacency matrix $A$, then $X$ admits instantaneous uniform mixing at time $\tau$ if 
\[ |U_A(\tau)|_{u,v} = \frac{1}{\sqrt{n}} \]
	for all vertices $u$ and $v$ of $X$. This is equivalent to the matrix $\sqrt{n}U_A(\tau)$ being a complex Hadamard matrix. Thus a necessary condition for instantaneous uniform mixing is that the Bose-Mesner algebra of $\mathcal A$ contains a complex Hadamard matrix. Chan gave the following result for association schemes containing a complex Hadamard matrix \cite[Proposition 2.2]{C}. 
	
\begin{theorem} \label{mixingbound}
	If the Bose-Mesner algebra of $\mathcal A$ contains a complex Hadamard matrix, then 
	\[ n \leq \Big( \sum_{j=0}^d |p_j(s)| \Big)^2. \]
\end{theorem}	

We can use this to give a necessary condition for instantaneous uniform mixing. Let $G$ be any finite group of order $n$, with conjugacy classes $K(g_0),K(g_1),...,K(g_d)$.
%ref1(8)
The columns of the idempotents $E_\chi$ are eigenvectors of $A_j$ with eigenvalues $\omega_\chi^j:=|K(g_j)|\chi(g_j)/\chi(1)$, so we can write $A_j$ as 
	\[ A_j =\sum_{\chi \in \Irr(G)} \omega_\chi^jE_\chi. \]
	We then deduce the following result.	
\begin{theorem}\label{mixingchartable}
	Let $G$ be a group of order $n$. If $\Cay(G,S)$ admits instantaneous uniform mixing, then the support of each $\chi \in \Irr(G)$ contains at least $\sqrt{n}$ elements of $G$.
\end{theorem}

\begin{proof}
	Suppose $\Cay(G,S)$ admits instantaneous uniform mixing. If $J$ is the set of indices $j$ such that $\chi(g_j) \neq 0$, then from Theorem~\ref{mixingbound} we have:
\begin{equation*}
	\begin{split}
		\sqrt{n} & \leq \sum_{j \in J} \frac{|K(g_j)||\chi(g_j)|}{\chi(1)}  \\
		& \leq \sum_{j \in J} |K(g_j)| \\
		&= |\Supp(\chi)|, \\
	\end{split}
\end{equation*}
	since $|\chi(g_i)| \leq \chi(1)$.
\end{proof}

Note that Theorem~\ref{mixingchartable} gives a necessary condition for instantaneous uniform mixing which depends only on the character table of $G$, and not on the connection set
%ref1(9)
$S$.  In the abelian case, there are Cayley graphs on abelian $2$-groups  and
elementary abelian $3$-groups (and direct products of these) that admit instantaneous
uniform mixing, as shown in  \cite{C} and  \cite{CF}. 

We have the following result. 

\begin{corollary}\label{mixing}
	There is no normal Cayley graph of an extraspecial $p$-group which admits instantaneous uniform mixing. 
\end{corollary}

\begin{proof}
	If $G$ is an extraspecial $p$-group of order $p^{2n+1}$ with center $Z$, then the nonlinear characters of $G$ vanish on $G\setminus Z$ and hence have supports of size $p < \sqrt{p^{2n+1}}$. Therefore Theorem~\ref{mixingchartable} implies that no normal Cayley graph of $G$ admits instantaneous uniform mixing. 
\end{proof}

\section{Conclusion}

Previous work on continuous-time quantum walks on Cayley graphs has mainly focused on the case where the group is abelian. In this paper we used known results on graphs in association schemes to determine when the quantum walks on Cayley graphs of extraspecial $2$-groups admit various phenomena. The choice to study the extraspecial $2$-groups is a natural one since previous results showed that perfect state transfer may only occur in a Cayley graph of a group containing a central involution, and each extraspecial $2$-group contains a \textit{unique} central involution. 

After giving precise conditions for perfect state transfer in these graphs, we saw that taking a connecting set $S$ such that $z \in S$ and $\ell$ even, or $z \notin S$ and $\ell$ odd are two particularly easy choices for $\Cay(G,S)$ to admit perfect state transfer, and we provided explicit examples of graphs admitting perfect state transfer when $z \notin S$ and $\ell$ is even by using partial spreads. We also gave precise conditions for the different forms of fractional revival, and showed that there are graphs from partial spreads which admit balanced fractional revival.   
		
We also remark that Cayley graphs of the group $(\Z/2\Z)^n$, known as \textit{cubelike graphs}, have received much attention previously \cite{BGS,CG,TFC}. As the quotient of an extraspecial group of order $2^{2n+1}$ by its center is isomorphic to $(\Z/2\Z)^{2n}$, it is interesting to compare the quantum walks on the Cayley graphs of these groups. For example, it was proven in \cite[Theorem 1]{BGS} that $\Cay((\Z/2\Z)^d,S)$ admits perfect state transfer at time $\pi/2$ when the sum of the elements of $S$ is not $0$. When this sum is $0$, Cheung and Godsil gave conditions for perfect state transfer in \cite[Theorem 4.1]{CG}, although our main result in Theorem~\ref{pstmain2} appears to be simpler. On the other hand, it is known that there are Cayley graphs of $(\Z/2\Z)^d$ and $(\Z/3\Z)^d$ which admit instantaneous uniform mixing. This starkly contrasts with our results in Corollary~\ref{mixing}. 

There are a few questions that we think would be interesting to investigate further:
\begin{enumerate}
\item Throughout the paper we only considered normal Cayley graphs, where the connecting set $S$ is a union of conjugacy classes. Are there similar results with more general connecting sets that are not unions of conjugacy classes? Perfect state transfer on non-normal Cayley graphs of dihedral groups has been studied in \cite{CCL}.%ref2
  
\item A finite $p$-group is called a \textit{special $p$-group} if its center, Frattini subgroup, and derived subgroup are all equal and elementary abelian. As the name suggests, the extraspecial $p$-groups are also special $p$-groups. Can the results presented here be generalized to Cayley graphs of special groups? Special groups which may be interesting to consider include the Heisenberg groups over finite fields. 
\end{enumerate}

\section*{Acknowledgements}
The authors would like to thank Ada Chan, Chris Godsil, Harmony Zhan, and Soffia Arnadottir for their helpful conversations on quantum walks, and for reading an early draft of this paper. Special thanks to Ada Chan for pointing out an important error in the first version as well as the connection between perfect state transfer in the graph and its complement, and to Chris Godsil for bringing the question of uniform mixing on Cayley graphs of extraspecial $p$-groups for odd $p$ to our attention. We also thank the anonymous referees
for their helpful comments and suggestions on the exposition. 
%\printbibliography


\begin{thebibliography}{99}

\bibitem{B} M. Ba\v{s}i\'{c}, \textit{Characterization of circulant networks having perfect state transfer.} Quantum Inf. Process. \textbf{12} (2013), no. 1, 345-364.  

%\bibitem{BCN} A.E. Brouwer, A.M. Cohen, and A. Neumaier, \textit{Distance-regular graphs}, Springer, 1989.

\bibitem{BGS} A. Bernasconi, C. Godsil, and S. Severini, \textit{Quantum networks on cubelike graphs}, Phys. Rev. A \textbf{78} (2008), no. 5, 052320.

\bibitem{CCL} X. Cao, B. Chen and S. Ling. \textit{ Perfect state transfer on Cayley graphs over dihedral groups, the non-normal case}, Electron. J. Combin. 27 (2020), no. 2, Paper No. 2.28, 18 pp. 81P45

\bibitem{CF} X. Cao, J. Wan, K. Feng. \textit{Some results on uniform mixing on abelian Cayley graphs}, arXiv:1911.07495 [math.CO], 2019.

\bibitem{CF2} X. Cao, K. Feng. \textit{Perfect state transfer on Cayley graphs over dihedral groups}. Linear Multilinear Algebr. \textbf{69} (2021), no. 2, 343-360.

\bibitem{C} A. Chan, \textit{Complex Hadamard matrices, instantaneous uniform mixing and cubes}, Algebr. Comb. \textbf{3} (2020), no. 3, 757-774.

\bibitem{CCTVZ} A. Chan, G. Coutinho, C. Tamon, L. Vinet, H. Zhan, \textit{Fractional revival and association schemes}, Discrete Math. \textbf{343} (2020), no. 11, 112018.

\bibitem{CJLSYZ} A. Chan, B. Johnson, M. Liu, M. Schmidt, Z. Yin, H. Zhan, \textit{Laplacian fractional revival on graphs}, arXiv:2010.10413 [math.CO].

\bibitem{CG}  W. Cheung and C. Godsil, \textit{Perfect state transfer in cubelike graphs}, Linear Algebra Appl. \textbf{435} (2011), 2468-2474

\bibitem{Co} G. Coutinho, \textit{Quantum state transfer in graphs}, Ph.D. thesis, University of Waterloo, 2014.

\bibitem{CGGV} G. Coutinho, C. Godsil, K. Guo, F. Vanhove, \textit{Perfect state transfer on distance regular graphs and association schemes.} Linear Algebra Appl. \textbf{478} (2015), 108-130.

\bibitem{D} P. Delsarte, \textit{An algebraic approach to the association schemes of coding theory}, Philips Res. Rep. Suppl. no. 10 (1973).

\bibitem{GW} H. Gerhardt and J. Watrous, \textit{Continuous-Time Quantum Walks on the
  Symmetric Group}, Lecture Notes in Comput. Sci., 2764, Springer, Berlin, 2003.

\bibitem{G} C. Godsil, \textit{Periodic graphs}, Electron. J. Comb. \textbf{18} (2011), no. 1. 

\bibitem{G2} C. Godsil, \textit{State transfer on graphs}, Discrete Math. \textbf{312} (2012), no. 1, 123-147.

\bibitem{GMR} C. Godsil, N. Mullin, A. Roy, \textit{Uniform mixing and association schemes}, Electron. J. Comb. \textbf{24}(\textbf{3}) (2017), no. 3.

\bibitem{GZ} C. Godsil and H. Zhan, \textit{Uniform mixing on Cayley graphs}, Electron. J. Comb. \textbf{24} (2017), no. 3. 

\bibitem{G3} D. Gorenstein, \textit{Finite groups}, AMS Chelsea Publishing Series, American Mathematical Soc. 2007.

\bibitem{M} B. Mortimer, \textit{The modular permutation representations of the known doubly transitive groups}, Proc. London Math. Soc. \textbf{s3-41} (1980), no.1, 1-20.

\bibitem{TFC} Y. Tan, K. Feng, X. Cao. \textit{Perfect state transfer on abelian Cayley graphs}, Linear Algebra Appl. \textbf{563} (2019), 331-352.
	
\end{thebibliography}
\end{document}